\documentclass{ws-ijnt}
\newtheorem{prop}[theorem]{Proposition}
\newcommand{\vep}{\varepsilon}

\newcommand{\lt}{\left}
\newcommand{\rt}{\right}
\newcommand{\mf}{\mathfrak}
\newcommand{\re}{\mf{Re}\,}
\newcommand{\im}{\mf{Im}\,}
\newcommand{\res}{\mbox{Res}}
\newcommand{\bal}{\begin{align}}
\newcommand{\bpm}{\begin{pmatrix}}
\newcommand{\epm}{\end{pmatrix}}
\newcommand{\bsm}{\lt(\begin{smallmatrix}}
\newcommand{\esm}{\end{smallmatrix}\rt)}
\newcommand{\beq}{\begin{equation}}
\newcommand{\eeq}{\end{equation}}
\newcommand{\hf}{\frac{1}{2}}
\newcommand{\kt}{\frac{k}{2}}
\newcommand{\kf}{\left(\frac{k}{4} - \hf\rt)}
\newcommand{\mellin}[2]{\int_0^\infty \!\!\! \mathfrak{f} \lt( #1\rt) y^{#2} \dy}

\renewcommand{\sp}{\ensuremath{s^\prime}}

\newcommand{\slfrac}[2]{\left.#1\right/#2}

\newcommand{\N}{\ensuremath{\mathbb{N}}}
\newcommand{\Q}{\ensuremath{\mathbb{Q}}}
\newcommand{\R}{\ensuremath{\mathbb{R}}}
\newcommand{\C}{\ensuremath{\mathbb{C}}}

\newcommand{\f}{\ensuremath{\mathfrak{f \,}}}
\newcommand{\addtoinf}[2]{\sum_{#1=1}^\infty #2}
\newcommand{\addtoinfcond}[3]{\sum_{\substack{#1 = 1\\ \ensuremath{#2}}}^\infty #3}
\newcommand{\addcond}[2]{\sum_{\ensuremath{#1}} \ensuremath{#2}}
\newcommand{\dterm}[3][\ensuremath{s}]{\frac{#2}{{#3}^{#1}}}
\newcommand{\runmod}[2]{\ensuremath{#1 \textrm{ mod } #2 }}
\newcommand{\kron}[2]{\lt( \frac{#1}{#2}\rt)}
\newcommand{\Lam}[4]{\Lambda\lt(\f_{#1}, \frac{#2}{#3} ,#4\rt)}

\newcommand{\dy}{\frac{\mathrm{d}y}{y}}

\begin{document}

\markboth{T. A. Hulse, E. M. K\i ral, C. I. Kuan \& L. Lim}
{The Sign of Fourier Coefficients of Cusp Forms}

\catchline{}{}{}{}{}
%
%%%%%%%%%%%%%%%%%%%%%%%%%%%%%%%%%%%%%%%%%%%%%%%%%%%%%%%%%%%%%%%%%%%%

\title{The Sign of Fourier Coefficients of Half-Integral Weight Cusp Forms
}

\author{Thomas A. Hulse, E. Mehmet K\i ral, Chan Ieong Kuan, Li-Mei Lim}

\address{Department of Mathematics, Brown University, 151 Thayer St.\\
Providence, RI 02912, USA \\ 
\email{\{tahulse,e.mehmet,ck9,llim\}@math.brown.edu}}

\maketitle

%\begin{history}
%\received{(Day Month Year)}
%\accepted{(Day Month Year)}
%\comby{xxx}
%\end{history}

\begin{abstract}
From a result of Waldspurger \cite{kz}, it is known that the normalized Fourier coefficients $a(m)$ of a half-integral weight holomorphic cusp eigenform $\f$ are, up to a finite set of factors, one of $\pm \sqrt{L(\frac{1}{2}, f, \chi_m)}$ when $m$ is square-free and $f$ is the integral weight cusp form related to $\f$ by the Shimura correspondence \cite{sh}. In this paper we address a question posed by Kohnen: which square root is $a(m)$?  In particular, if we look at the set of $a(m)$ with $m$ square-free, do these Fourier coefficients change sign infinitely often? By partially analytically continuing a related Dirichlet series, we are able to show that this is so. 
\end{abstract}

%\keywords{cusp form; half-integral weight; sign change}
%
%\ccode{Mathematics Subject Classification 2010: 11F37, 11F30}

\section{Introduction} 

Let $k$ be an odd integer and $\f \in S_{\kt} (\Gamma_0(4))$; that is, a cusp form of half-integral weight $k/2$ and level $4$ as described by Shimura \cite{sh}.
% such that 
%\[
%a(m) = 0 \mbox{ for } (-1)^{\kt - \hf}m \equiv 2,3 \mod 4.
%\]
We will discuss the automorphic properties of $\f$ in more detail in Section \ref{aut3}. Let the Fourier expansion of $\f$ at $\infty$ be
\beq
\f(z) = \sum_{m=1}^\infty a(m) m^{\frac{k}{4} - \hf}e(mz).
\eeq

The Shimura correspondence provides a holomorphic modular form $f$ of weight $k-1$ such that if $\f$ is an eigenform of the Hecke operator $T_{\kt} (p^2)$, then $f$ is an eigenform of $T_{k-1}(p)$ with the same eigenvalue. It was proven in \cite{nw} that the level of $f$ is 2. For the definitions of the Hecke operators in the half-integral weight case, see \cite{sh}. Waldspurger proved in \cite{w} that for any square-free $t$,
%Let $\f$ be a cusp form of half-integral weight and level $4$, as described by \cite{sh}. More precisely, $\f \in S^+_\frac{k}{2}(\Gamma_0(4))$, where $k$ is an odd integer; we will discuss the automorphic properties of $\f$ in more detail in Section \ref{aut3}. Let
%\beq
%\f(z) = \sum_{m=1}^\infty a(m) m^{\frac{k}{4} - \hf}e(mz) 
%\eeq
%be the Fourier expansion of $\f$ at $\infty$, where $e(z) = e^{2\pi i z}$. We can scale $\f$ such that these $a(m)$ are real, and since $\f\in S^+_\frac{k}{2}(\Gamma_0(4))$ we also have that $a(m)=0$ for $m\equiv 2,3 \pmod{4}$. The Shimura correspondence gives us a %holomorphic cusp form $f$ of integral weight $k-1$ such that the $T_{n^2}$ Hecke eigenvalue on $\f$ agrees with the $T_n$ Hecke eigenvalue of $f$. A theorem of Waldspurger\cite{kz} further demonstrates that for a fundamental discriminant $D$ such that $(-1)^{\kt-\hf} %D>0$, and $\chi_D(n)=\lt(\frac{D}{n}\rt)$ the Kronecker symbol 
\beq \label{prop}
a(t)^2 =c\cdot L\lt(\hf,f,\chi_t\rt)
\eeq
where 
\beq \label{chi}
\chi_t(n) = \lt(\frac{(-1)^{\kt - \hf}t}{n}\rt)
\eeq
is the unique real primitive character modulo $t$. The constant $c$ is dependent only on $\f$. Later this result was made explicit by Kohnen and Zagier in the case of $\f \in S^+_{k+\hf}(\Gamma_0(4))$, see \cite{kz} for the definition of the space $S^+_{k+\hf}(\Gamma_0(4))$. For a fundamental discriminant $D$ satisfying
\beq
(-1)^{\kt - \hf} D > 0,
\eeq
they proved that
\beq \label{KohnenZagier}
\frac{a(|D|)^2}{\langle \f,\f\rangle} = \frac{(\frac{k-1}{2}-1)!}{\pi^{(k-1)/2}}\frac{L(\hf,f,\chi_D)}{\langle f,f \rangle}.
\eeq
Here $\langle \f,\f \rangle$ and $\langle f, f\rangle$ are the normalized Petersson inner products, and
\beq
\chi_D (n) = \kron{D}{n}
\eeq
is the Kronecker symbol. 

The relationship \eqref{prop} prompts the questions posed by Kohnen: which square root of $L\lt(\hf,f,\chi_t\rt)$ is $a(t)$ proportional to, and how often? In \cite{k}, Kohnen in fact proves that for any half-integral weight cusp form $\f\in S_{k+\frac{1}{2}}(N, \chi)$, not necessarily an eigenform, the sequence of Fourier coefficients $a(tn^2)$ for a fixed $t$ square-free has infinitely many sign changes.  A natural next question one may ask is whether all the Fourier coefficents $a(t)$ with $t$ running over square-free integers change sign infinitely often. In the following theorem we prove that this is indeed the case for eigenforms. 
\begin{theorem}\label{mainthm}
Given $\f \in S_\frac{k}{2}(\Gamma_0(4))$, an eigenform of all Hecke operators $T_\kt(p^2)$ for $p$ prime, where $k$ is an odd integer, with Fourier expansion 
\beq
\f(z) = \sum_{m=1}^\infty a(m) m^{\frac{k}{4} - \hf}e(mz); 
\eeq
the Fourier coefficients $a(t)$, with $t$ running over square-free integers, change sign infinitely often.

\end{theorem}

Inspired by the methods in \cite{gh} and \cite{hr}, we will prove Theorem \ref{mainthm} by analytically continuing the Dirichlet series 
\beq \label{sqfree}
\sum_{\substack{t\geq 1\\t\textrm{ square-free}}}\frac{a(t)}{t^s}
\eeq
to $\Re(s)>3/4$ by exploiting the analytic continuations of a family of Mellin transforms related to $\f$; we then prove our claim by contradiction. 

\section{Automorphic Properties}\label{aut3}

Before we proceed, we will review the automorphic properties of half-integral weight cusp forms. Let $\f$ be as above. Given $\gamma=\bsm a& b \\ c& d\esm \in \Gamma_0(4)$, we have

\beq 
j(\gamma,z) := \vep_d^{-1} \kron{c}{d}(cz+d)^{\hf}  = \theta(\gamma(z))/\theta(z)
\eeq
where $\kron{c}{d}$ is Shimura's extension of the Jacobi symbol as in \cite{sh}.  Setting $\xi := \lt(\gamma, j(\gamma, z) \rt)$, $\f$ satisfies
\beq 
\f|_{\lt[\xi \rt]_k}(z) := j(\gamma,z)^{-k} \f(z) = \vep_d^k \kron{c}{d}(cz+d)^{-\kt}\f\lt( \bpm a & b \\ c & d  \epm z \rt)= \f(z). \label{modular1}
\eeq
Here $\vep_d$ is the sign of the Gaussian sum $\Sigma_{n = 1}^d e(\frac{n^2}{d})$: 
\beq
\vep_d = \lt\{ \begin{array}{ll}
1 & \mbox{ if } d \equiv 1 \pmod{4} \\
i  & \mbox{ if } d \equiv 3 \pmod{4}.
 \end{array}\rt.
\eeq
Furthermore, we fix the following expressions for $\f_\hf$ and $\f_0$, which are evaluations of $\f$ at the respective cusps $\hf$ and $0$, as
\begin{align}
\f_{\hf}(z) &:= \f|_{\lt[\bsm 1 & 0 \\ -2 & 1 \esm \rt]_k}(z) =(-2z+1)^{-\kt}\f\lt( \bpm 1 & 0 \\ -2 & 1 \epm z \rt)\label{modular2}\\
\f_{0}(z) &:= \f|_{\lt[\bsm 0 & -1 \\ 4 & 0 \esm \rt]_k}(z) =(-2iz)^{-\kt}\f\lt( \bpm 0 & -1 \\ 4 & 0 \epm z \rt).
\label{modular3}
\end{align}
Also note that $\f\lt(\bsm r & 0 \\ 0 & r \esm z\rt)=\f(z)$ for all $r\in\R^*$.

\section{Arguing by Contradiction}
Our proof by contradiction proceeds as follows. Take the Dirichlet series 
\beq
M(s)=\sum_{\substack{t\geq 1\\t\textrm{ square-free}}}\frac{a(t)}{t^s}
\eeq
as described in \eqref{sqfree} and assume that $a(t)$ changes sign finitely many times. Assume for a contradiction that $a(t)\geq 0$ for $t > T$ where $T$ is sufficiently large.  Throughout this section, we let $t$ denote a square-free positive integer.

Suppose that $M(s)$ analytically continues to $\Re(s)>\frac{3}{4}$ with polynomial growth in $\Im(s)$, as this work will demonstrate. Using a well-known inverse Mellin transform, we get 
\begin{align} \label{mel1}
\frac{1}{2\pi i}\int_{2-i\infty}^{2+i\infty}M(s)\Gamma(s)x^s\ ds &= \sum_{t} a(t)e^{-t/x}.
\end{align} 
The integral on the left-hand side above is $O(x^{3/4+\varepsilon})$ for any $\vep>0$, as we can move the line of integration to $\Re(s)=3/4+\varepsilon$.  On this vertical line, the gamma function decreases exponentially, whereas the analytic continuation of $M(s)$ only has polynomial growth, as will be shown below in Proposition \ref{analcont}. Since the integrand has no poles for $\Re(s)>3/4$, we don't pick up any residues in moving the line of integration.  Thus we arrive at the inequality
\beq
\label{upperbound}
\sum_{t} a(t) e^{-t/x} \ll x^{3/4+\vep}.
\eeq

The completed Eisenstein series for level 4 is
\beq
E^*(z,s)=2^{2s-1}\zeta^*(2s)E(z,s) =2^{2s-1} \zeta^*(2s)\sum_{\gamma \in \Gamma_\infty \backslash \Gamma_0(4)} \Im(\gamma z)^s,
\eeq 
where $\zeta^*(s)$ is given by
\beq
\zeta^*(s)=\pi^{-s/2}\Gamma(s/2)\zeta(s).
\eeq
The completed Eisenstein series is holomorphic except for simple poles at $s = 0,1$, with residues $ -\frac{1}{4}$ and $\frac{1}{4}$ respectively. We have the identity
\beq
 \iint_{\Gamma_0(4) \backslash \mathfrak{h}} |\f(z)|^2 E^*(z,s) y^{k/2} \frac{\mathrm{d}x \mathrm{d}y}{y^2} = \Gamma\lt(s + \kt -1\rt) 2^{1-k}\pi^{-(s + \kt -1)} \zeta^*(2s) L^{(2)}(\f, s), 
\eeq
where
\beq
L^{(2)}(\f, s)=\sum_{m=1}^\infty \frac{a(m)^2}{m^s}
\eeq
which follows after a Rankin-Selberg unfolding. This implies that $L^{(2)}(\f, s)$ has a pole at $s = 1$ with a non-zero residue.  In fact, due to the integral representation above, the Rankin-Selberg convolution $L$-series extends to a meromorphic function with the only pole at $s = 1$ when $\re(s)\geq \hf$.  

Considering the inverse Mellin transform
\beq \label{biggerx}
I=\frac{1}{2\pi i}\int_{2 - i \infty}^{2 + i \infty} L^{(2)}(\f, s)  \Gamma(s) x^s \mathrm{d}s =\sum_{m} a(m)^2e^{-m/x},
\eeq
and shifting the line of integration to $\Re(s) = \hf$ past the pole at $s=1$, we get 
\beq
I=(\res_{s=1}L^{(2)}(\f, s))x+\frac{1}{2\pi i}\int_{\hf-i\infty}^{\hf+i\infty} L^{(2)}(\f, s)  \Gamma(s) x^s \mathrm{d}s,
\eeq
which implies, as the contribution from the integral above is $O(x^{\hf})$, that
\beq
x \ll \sum_{m} a(m)^2e^{-m/x}.
\eeq

In the above sum, the square-free integers play a nontrivial role.
%In fact we will show the inequality 
%\[
% x \ll \sum_{t} a(t)^2e^{-t/x}.
% \] 
Indeed, using Lemma \ref{bound} we will be able to conclude that for any $\vep >0$ and $\Re(s) = \sigma$

\bal \notag
L^{(2)}(\mf{f},s)&=\sum_{m=1}^\infty \frac{a(m)^2}{m^s} = \sum_t \sum_{n=1}^\infty \frac{a(tn^2)^2}{(tn^2)^s} \\
& \ll
\sum_t \sum_{n=1}^\infty \frac{a(t)^2}{t^\sigma n^{2\sigma - 2\vep}} = \zeta(2\sigma - 2\vep) M^{(2)}(\sigma) \label{3.5}
\end{align}
where
\beq
M^{(2)}(s)=\sum_t \frac{a(t)^2}{t^s}.
\eeq
Therefore, as we move $\sigma$ towards $1^+$ along the real line, since $L^{(2)}(\f,s)$ has a pole at $1$ and $\zeta(2\sigma-2\vep)\ll 1$ for $\sigma\geq 1$, it follows from \eqref{3.5} that $\sum_t \slfrac{a(t)^2}{t^\sigma}$ tends to infinity. Therefore the function $M^{(2)}(s)$, which is an analytic function on the region $\Re(s) >1$, has a singularity as $s$ tends to $1$. Put 
\beq
A(x) = \sum_{T < t \leq x} a(t)^2.
\eeq
We claim that the singularity of $M^{(2)}(s)$ forces the partial sums, $A(x)$, to not be of slow growth. Indeed, assume that for some $c < 1$,
\beq
A(x) = O(x^c),
\eeq
from the partial summation formula we get, for $\Re(s)>1$
\beq
\sum_{t >T} \frac{a(t)^2}{t^s} = s \int_T^\infty \frac{A(u)}{u^{s+1}} \mathrm{d}u.
\eeq
Then, since we assumed $A(x)\ll x^c$, the right-hand side above is an analytic function on the right half plane $\Re(s) > c$, but the left-hand side $M^{(2)}(s)$ has a singularity at $s=1$, giving our contradiction. Thus, for every $c$ with $0<c<1$, every constant $\alpha>0$, and every $x$, there is an $x_0> x$ such that 
\beq \label{channotbigoh}
A(x_0) \geq \alpha x_0^c.
\eeq

Now from $\eqref{upperbound}$ we have a constant $\beta>0$ such that,
\begin{align}
&\beta x^{3/4 +\vep} \geq  e \lt| \sum a(t)e^{-t/x}\rt| \geq e\lt| \sum_{t > T} a(t)e^{-t/x} \rt|  - e\lt| \sum_{t \leq T} a(t)e^{-t/x} \rt|, 
\end{align}
so using our assumption on the eventual non-negativity of $a(t)$, we have that
\bal
& \beta x^{3/4 + \vep} + O(1) \geq e \sum_{t > T} a(t)e^{-t/x} \geq \sum_{T< t < x} a(t),
\end{align}
where $\vep >0$ is arbitrarily small, and $\beta$ depends on $\vep$. Thus increasing $\beta$ to accommodate the constant term, we get
\beq
\beta x^{3/4 + \vep} \geq \sum_{T< t < x} a(t).
\eeq

Let $\lambda n^\theta$ be an upper bound on the individual coefficient $a(n)$ of the half-integral weight modular form $\f$; according to \cite{iw} one may take $\theta = 3/14.$ Now apply \eqref{channotbigoh} and get that for some $x_0$ as above, which we may choose to be arbitrarily large,
\beq \label{ineqtrain}
\lambda \beta x_0^{3/4 + \vep} \geq \lambda\sum_{T < t\leq x_0} a(t) \geq \sum_{T < t \leq x_0} \frac{a(t)^2}{t^\theta} \geq x_0^{-\theta} \sum_{T< t \leq x_0} a(t)^2 \geq \alpha x_0^{-\theta + c},
\eeq
again by our non-negativity assumption. This implies,
\beq
x_0^{c - \vep -27/28} \leq \frac{\lambda \beta}{\alpha}.
\eeq

By chosing $c$ and $\vep$ appropriately we may make the exponent on the left hand side greater than $0$, giving our contradiction. Therefore, the assumption that all Fourier coefficients $a(t)$ change sign finitely many times for square-free $t$ must be false. Thus, in order to show that the Fourier coefficients $a(t)$ change sign infinitely often for square-free $t$, we need only show that $M(s)$ can be analytically continued up to the line $\Re(s) = 3/4$, and grows slowly on vertical strips.  The remainder of this paper is devoted to proving this.

\section{Analytic Continuation}

We now proceed to obtain an analytic continuation of the Dirichlet series \eqref{sqfree} to the region $\Re (s) > 3/4$. First, note that 

\beq \label{first}
\sum_{\substack{t\geq 1\\t\textrm{ square-free}}}\frac{a(t)}{t^s} 
= \addtoinf{m}{\frac{a(m)}{m^s}} \addcond{r^2|m}{\mu(r)}
= \addtoinf{r}{\mu(r)} D_r(s),
\eeq
where 
\beq
D_r(s)=\addtoinfcond{m}{m \equiv 0 \textrm{ mod } r^2}{\frac{a(m)}{m^s}}.
\eeq

Lemma \ref{oneoverrsq} shows that the series $D_r(s)$ converges for $\Re(s) > 1$, and in fact $D_r(1 + \vep + it) \ll_\vep 1/r^2$. With this fact, we easily see that our Dirichlet series over the square-free integers converges on the half-plane $\Re(s) > 1$. We now further examine the series $D_r(s)$:

\begin{align}
D_r(s) &= \addtoinfcond{m}{m \equiv 0 \textrm{ mod } r^2}{\frac{a(m)}{m^s}} = \addtoinf{m}{\dterm{a(m)}{m}} \left( \frac{1}{r^2} \!\! \addcond{u \textrm{ mod } r^2}{e\lt(\frac{mu}{r^2}\rt)} \rt) \notag \\
&= \dterm[2]{1}{r} \addcond{ \runmod{u}{r^2}}{} \addtoinf{m}{\dterm{a(m)e(\frac{mu}{r^2})}{m}}. 
\end{align}

The innermost Dirichlet series can be expressed in terms of an additively twisted Mellin integral of \f. For rational $q$, denote 
\beq
\Lambda(\f_*,q,s) = \mellin{iy + q}{s + \kf}.
\eeq
This integral converges for every $s \in \C$ because $q \in \Q$ is a cusp and since $\f$ is a cusp form, $\f(iy + q)$ decreases exponentially as $y \to \infty$ and as $y \to 0$. Thus $\Lambda(\f, q, s)$ is an entire function of $s$. Let $\f_*$ denote $\f,\f_\hf,$ or $\f_0$. Now expanding $\f_*$ in the integral as its Fourier series, with respective coefficients $a_*(n)$, we get:
\begin{eqnarray}
\int_0^\infty \f_*(iy+q)y^{s+\kf}\dy &=& \int_0^\infty \addtoinf{m}{a_*(m)m^{\kf} e(m(iy + q)) y^{s + \kf} \dy}\notag \\
&=& \addtoinf{m}{a_*(m)m^{\kf}e(m q)} \int_0^\infty \!\!\! e^{- 2\pi m y} y^{s + \kf} \dy \notag\\
&=& \frac{ \Gamma(s + \lt(\tfrac{k}{4} - \tfrac{1}{2}\rt))}{(2\pi)^{s + \kf}} \addtoinf{m} \frac{a_*(m)e(m q)}{m^s} .\label{polyb}
\end{eqnarray}

For ease of notation call, from now on, $\sp : = s + \kf$. Using the integral representation of our Dirichlet series, which is that of $L(\f, s)$ twisted by an additive character, we obtain
 
\begin{align} \label{main}
D_r(s) &=  \addtoinfcond{m}{m \equiv 0 \textrm{ mod } r^2}{\frac{a(m)}{m^s}} \notag\\
&= \frac{(2\pi)^{\sp}}{\Gamma(\sp)}\frac{1}{r^2} \sum_{\runmod{u}{r^2}} \Lam{}{u}{r^2}{s} \notag \\
&= \frac{(2\pi)^{\sp}}{\Gamma(\sp)} \frac{1}{r^2} \addcond{d|r^2}{} \addcond{\substack{(u,d)=1\\ \runmod{u}{d}}}{ \Lam{}{u}{d}{s}}, 
\end{align}
where the fraction $u/d$ is in lowest terms, and by abuse of notation, we continue to call the numerator $u$. Equation \eqref{main} allows us to express $D_r(s)$ as a finite sum of entire functions, hence $D_r(s)$ itself is an entire function. Therefore it makes sense to talk about the growth properties of $D_r(s)$ in $r$ for any fixed $s$ on the complex plane. Also note that we only need to estimate $D_r(s)$ for square-free $r$, due to the existence of $\mu(r)$ in \eqref{first}.

\begin{lemma}\label{bound}
The Fourier coefficients of a half-integral weight cusp eigenform $\f\in S_\frac{k}{2}(\Gamma_0(4))$, with $k \geq 5$ as above satisfy the following bound
\beq
a(tn^2) \ll |a(t)|n^\vep
\eeq
where $t,n \in \N$ and $t$ is square-free, with the implied constant dependent only on $\f$ and $\vep>0$.
\end{lemma}
\begin{proof}
The Shimura correspondence \cite{sh} gives us that for $t \in \N$ square-free we have that
\beq \label{4.5}
\addtoinf{n} \frac{a(tn^2)}{n^{s-\kt+1}}=a(t)\lt( \addtoinf{m_1}\frac{\chi_t(m_1)\mu(m_1)}{m_1^{s-\kt+\frac{3}{2}}} \rt)\lt(\addtoinf{m_2} \frac{A(m_2)}{m_2^s}\rt),
\eeq
where $\chi_t(m_1)=\kron{-1}{m_1}^{\kt-\hf}\kron{t}{m_1}$ and the $A(m_2)$ are the Fourier coefficients of the weight $k-1$ cusp form $f \in S_{k-1}(\Gamma_0(2))$ such that
\beq
f(z)=\sum_{m_2=1}^\infty A(m_2)e^{2\pi i m_2 z}
\eeq
where $f$ is associated to $\f$ by the Shimura correspondence. Expanding the right-hand side of \eqref{4.5} we get
\bal
\addtoinf{n} \frac{a(tn^2)}{n^{s-\kt+1}} &= a(t) \addtoinf{n} \  \sum_{m_1m_2=n} \frac{\chi_t(m_1)\mu(m_1)}{m_1^{s-\kt+\frac{3}{2}}} \frac{A(m_2)}{m_2^s} \notag \\ 
&= a(t) \addtoinf{n} \sum_{m|n}  \frac{\chi_t\lt(\frac{n}{m}\rt)\mu\lt(\frac{n}{m}\rt)}{\lt(\frac{n}{m}\rt)^{s-\kt+\frac{3}{2}}} \frac{A(m)}{m^s} \notag \\
&=a(t)  \addtoinf{n} \frac{1}{n^{s-\kt+1}} \sum_{m|n}  \frac{\chi_t\lt(\frac{n}{m}\rt)\mu\lt(\frac{n}{m}\rt)}{n^{\hf}} \frac{A(m)}{m^{\kt-\frac{3}{2}}}. 
\end{align}
Comparing coefficients term-by-term we see that for each $n$ 
\bal
a(tn^2)=a(t)\sum_{m|n}  \frac{\chi_t\lt(\frac{n}{m}\rt)\mu\lt(\frac{n}{m}\rt)}{n^{\hf}} \frac{A(m)}{m^{\kt-\frac{3}{2}}}. \label{preb}
\end{align}

Since the Ramanujan-Petersson conjecture is known for integral weight cusp forms, we have $A(m)\ll m^{\frac{(k-1)-1}{2}+\vep}$ with the implied constant dependent on $f$ (and thus $\f$) and $\vep$. Using this bound and taking absolute values of \eqref{preb} we get 
\beq
a(tn^2) \ll |a(t)|n^{-\hf} \sum_{m|n} m^{\hf+\vep} \ll |a(t)|n^{\vep-\hf} \sigma_{\hf}(n). \label{bou}
\eeq
Since  $\sigma_{\hf}(n) \leq d(n)\sqrt{n}$, where $d(n)$ is the divisor function, we have $\sigma_{\hf}(n) \ll n^{1/2+\vep}$ with the implied constant dependent on $\vep$. Putting this into \eqref{bou} gives the desired result.
\end{proof}

\begin{lemma} \label{oneoverrsq}
Letting $r\in\N$ and $\tau \in \R$,
\beq \label{lemma1eq}
D_r(1 + \vep + i\tau)  = \addtoinfcond{m}{m \equiv 0 \textrm{ mod } r^2}{\frac{a(m)}{m^{1+\vep+i\tau}}} \ll \frac{1}{r^2}
\eeq
where the implied constant depends only on $\f$ and $\vep$.  
\end{lemma}
\begin{proof}
In the sum, we write $m = nr^2$, and let $n_0$ be the square-free part of $n$. Then by Lemma \ref{bound} $a(nr^2) \ll |a(n_0)|\lt(\frac{nr^2}{n_0}\rt)^\vep$, and therefore for $s=\sigma+i\tau$ with $\sigma \geq 1$, 
\beq
\addtoinf{n} \frac{a(nr^2)}{n^{s+2\vep}}  \ll r^{2\vep}\addtoinf{n} \frac{|a(n_0)|(n/n_0)^\vep}{n^{\sigma+2\vep}} \leq r^{2\vep}\addtoinf{n}\frac{|a(n_0)|}{n^{\sigma+\vep}}, \label{imtrappedinatexfactory}
\eeq
where the implied constant only depends on $\f$ and $\vep$. Now, 
\bal
\addtoinf{n}\frac{|a(n_0)|}{n^{\sigma+\vep}} &= \notag  \sum_{d=1}^\infty \sum_{\substack{ n \\ \mbox{\tiny square-free}}} \frac{|a(n)|}{(nd^2)^{\sigma+\vep}}\\
& = \zeta(2\sigma+2\vep)\sum_{\substack{ n \\ \mbox{\tiny square-free}}} \frac{|a(n)|}{n^{\sigma+\vep}}  \ll \zeta(2\sigma+2\vep)\sum_{n=1}^\infty \frac{|a(n)|}{n^{\sigma+\vep}} = O(1) \label{fix}
\end{align}
where the implied constant again depends on $\f$ and $\vep$.  Therefore, using both \eqref{imtrappedinatexfactory} and \eqref{fix} we have
\bal
  D_r(s+\vep) &= \addtoinfcond{m}{m \equiv 0 \textrm{ mod } r^2}\frac{a(m)}{m^{s+\vep}} = \addtoinf{n}\frac{a(nr^2)}{(nr^2)^{s+\vep}} \notag \\
  &= \frac{1}{r^{2s+2\vep}} \addtoinf{n}{\frac{a(nr^2)}{n^{s+\vep}}} \ll \frac{1}{r^{2\sigma}}.
\end{align}
Letting $s = 1+ i\tau$, we have
\beq
D_r(1+\vep+i\tau)  \ll  \frac{1}{r^2}, 
\eeq
completing our proof.
\end{proof}

In order to show that 
\beq M(3/4 + \vep) = \sum \mu(r)D_r(3/4 + \vep)\eeq 
converges, we will bound $D_r(-\vep + it)$ by a polynomial in $t$ and apply a Phragm\'{e}n-Lindel\"{o}f convexity argument. 

The twisted Mellin integrals $\Lambda(\f, \frac{u}{d}, s)$ have functional equations. Depending on the class of equivalent cusps that $\frac{u}{d}$ belongs to, we get slightly different functional equations.  They are as follows:

\begin{lemma}\label{class}
If $4|d$, 
\beq
\Lam{}{u}{d}{s} = d^{1-2s}(-i)^\kt \vep_v^k\kron{d}{v}  \Lam{}{v}{d}{1-s} ,
\eeq
where $v$ is chosen so that $uv \equiv -1 \pmod{d}$. 

If $2|| d$, the functional equation has the type,
\beq
\Lam{}{u}{d}{s} = d^{1-2s}(-i)^\kt \vep_v^k\kron{d}{v} \Lam{\hf}{v}{d}{1-s},
\eeq
where once again $v$ is chosen to satisfy $uv \equiv -1 \pmod{d}$.

Finally, if $2 \nmid d$, then
\beq
\Lam{}{u}{d}{s} = (2d)^{1-2s} \, \vep_d^{-k} \kron{v}{d}  \Lam{0}{v}{d}{1-s}.
\eeq
where $v$ be such that $4uv \equiv -1 \pmod{d}$.
\end{lemma}

\begin{proof}
First note that 
\beq
\mellin{iy+\frac{u}{d}}{s + \kf} = \mellin{\bpm 1& u/d \\ 0 &1 \epm iy}{s+\kf}.
\eeq
We observe that $u/d$ is equivalent to to the cusps $\infty, 1/2$, or $0$ depending on the conditions  $4|d, 2\|d$ or $2\nmid d$ respectively. We consider the following matrix decompositions in each case. 
If $4|d$,
\beq
\bpm 1& u/d \\ 0 & 1\epm \bpm 4/d& 0 \\ 0 & d \epm = \bpm -u &e\\ -d & v \epm  \bpm 1& v/d \\ 0 &1 \epm \bpm 0 & -1 \\ 4 & 0 \epm  \label{id1},
\eeq
and if $2\| d$, 
\beq
 \bpm 1& u/d \\ 0 & 1\epm \bpm 4/d& 0 \\ 0 & d \epm = \bpm 2e-u &e\\ 2v-d & v \epm  \bpm 1&0 \\ -2&1 \epm \bpm 1& v/d \\ 0 &1 \epm \bpm 0 & -1 \\ 4 & 0 \epm  \label{id2},
\eeq
where $v$ and $e$ are chosen to satisfy $uv - de = -1$.
Finally for $d$ odd, with $v$ and $e$ chosen to satisfy $4uv - de = -1$,  we consider the following matrix decomposition:
\beq
\bpm 1& u/d \\ 0 & 1\epm \bpm 1/d& 0 \\ 0 & d \epm = \bpm e & u\\ 4v & d \epm  \bpm 0&1/4 \\ -1&0 \epm \bpm 1& v/d \\ 0 &1 \epm \bpm 0 & -1 \\ 4 & 0 \epm. \label{id3}
\eeq
Note that in all of the matrix decompositions, the leftmost matrix is in $\Gamma_0(4)$.

Recall that we let $s'=s+\kf$. For $4|d$ we use \eqref{id1},
\begin{align}
\Lambda &\lt(\f,\frac{u}{d},s\rt) 
= \mellin{iy + \frac{u}{d}}{s+\kf} \notag \\
&= \mellin{\bpm 1 & u/d \\ 0&1 \epm \bpm 4/d&0\\ 0 & d \epm \frac{id^2y}{4}}{s'} \notag \\
&= 4^{\sp}d^{-2\sp} \mellin{ \bpm -u &e\\ -d & v \epm  \bpm 1& v/d \\ 0 &1 \epm \bpm 0 & -1 \\ 4 & 0 \epm iy}{s'} \notag\\
%&=4^{\sp}d^{-2\sp} \mellin{ \bpm -u &e\\ -d & v \epm  \bpm 1& v/d \\ 0 &1 \epm \frac{-1}{4iy}}{s + \kf} \notag \\
%&=d^{-2\sp} \mellin{ \bpm -u &e\\ -d & v \epm  \bpm 1& v/d \\ 0 &1 \epm iy}{-s - \kf} \notag\\
&=d^{-2\sp} \mellin{\bpm -u & e \\ -d &v \epm \lt( iy + \frac{v}{d} \rt)}{-s'} \notag \\
&=d^{-2\sp} \int_0^\infty \!\!\! \f|_{[\xi]_k} \lt( iy + \frac{v}{d} \rt) \vep_v^{-k} \kron{-d}{v} \lt(-d\lt(iy+\frac{v}{d}\rt) + v\rt)^{\kt} y^{-s'} \dy \notag \\
%&=d^{-2\sp} \int_0^\infty \!\!\! \f|_{[\xi]_k} \lt( iy + \frac{v}{d} \rt) \vep_v^{-k} \kron{-d}{v} (-diy)^{\kt} y^{-s - \kf} \dy \notag \\
%&=d^{\kt - 2\sp}(-i)^{\kt} \vep_v^k \kron{d}{v}  \mellin{iy + \frac{v}{d}}{\frac{k}{2} - s - \kf} \notag \\
&=d^{1-2s}(-i)^\kt \vep_v^k \kron{d}{v} \mellin{iy + \frac{v}{d}}{1 - s + \kf},
 \end{align}
where $\xi=\lt(\gamma, j(\gamma,z)\rt)$ for $\gamma =  \bsm -u & e \\ -d & v \esm$ and $\f|_{[\xi]_k}$ denotes the slash operator on half-integral weight forms as described in \eqref{modular1}. Thus, for $4|d$,

\beq \label{funceqfour}
\Lam{}{u}{d}{s} = d^{1-2s}(-i)^\kt \vep_v^k\kron{d}{v} \Lam{}{v}{d}{1-s}.
\eeq

If $2||d$, once again let $v$ and $e$ satisfy $uv - de = -1$.  By \eqref{id2} and similar reasoning as above, and using \eqref{modular1} and \eqref{modular2}, we deduce:

\begin{align}
\Lambda &\lt(\f,\frac{u}{d},s\rt) = \mellin{iy + \frac{u}{d}}{s + \kf} \notag \\
%&= \mellin{\bpm 1 & u/d \\ 0&1 \epm \bpm 4/d&0\\ 0 & d \epm \frac{id^2y}{4}}{s + \kf} \notag \\
%&= 4^{\sp} d^{-2\sp} \mellin{ \bpm 2e-u &e\\ 2v-d & v \epm \bpm 1& 0 \\ -2 &1 \epm  \bpm 1& v/d \\ 0 &1 \epm \frac{-1}{4iy}}{s + \kf} \notag \\
&= d^{-2\sp} \mellin{ \bpm 2e-u &e\\ 2v-d & v \epm \bpm 1& 0 \\ -2 &1 \epm  \bpm 1& v/d \\ 0 &1 \epm iy}{-s'} \notag \\
&=d^{-2\sp} \int_0^\infty \!\!\! \f|_{[\xi]_k}|_{\lt[\bsm 1 & 0 \\ -2 &1 \esm \rt]_k} \lt( iy + \frac{v}{d} \rt) \vep_v^{-k} \kron{2v-d}{v}^k (-diy)^{\kt} y^{-s'} \dy \notag \\
&=d^{1-2s}(-i)^{\kt} \vep_v^k \kron{d}{v}\int_0^\infty \!\!\! \f_{\hf} \lt(iy + \frac{v}{d}\rt)y^{1 - s + \kf} \dy 
\end{align}
where, this time, $\gamma= \bsm 2e-u & e \\ 2v-d & v \esm$, and $\xi = \lt( \gamma, j(\gamma,z) \rt)$. Thus for $2||d$, 
\beq
\Lam{}{u}{d}{s} = d^{1-2s}(-i)^\kt \vep_v^k\kron{d}{v} \Lam{\hf}{v}{d}{1-s}.
\eeq

For $d$ odd, we choose $v,e$ such that $4uv-de=-1$. So by \eqref{id3},  \eqref{modular1} and \eqref{modular3} we have

\begin{align}
\Lambda&\lt(\f,\frac{u}{d},s\rt)= \mellin{iy + \frac{u}{d}}{s+\kf} \notag \\
%&= \mellin{\bpm 1 & u/d \\ 0&1 \epm \bpm 1/d&0\\ 0 & d \epm id^2y}{s + \kf} \notag \\
%&= d^{-2\sp} \mellin{\bpm e &u\\ 4v & d \epm \bpm 0 & 1/4 \\ -1 & 0 \epm  \bpm 1& v/d \\ 0 &1 \epm \bpm 0 & -1 \\ 4 & 0 \epm iy}{s +\kf} \notag\\
%&=d^{-2\sp} \mellin{ \bpm e &u\\ 4v & d \epm \bpm 0 & 1/4 \\ -1 & 0 \epm  \bpm 1& v/d \\ 0 &1 \epm \frac{-1}{4iy}}{s + \kf} \\
%&=(2d)^{-2\sp} \mellin{\bpm e &u\\ 4v & d \epm \bpm 0 & 1/4 \\ -1 & 0 \epm  \bpm 1& v/d \\ 0 &1 \epm  iy}{-s- \kf}\notag \\
&=(2d)^{-2\sp} \mellin{\bpm e &u\\ 4v & d \epm \bpm 0 & -1 \\ 4 & 0 \epm \lt( iy +\frac{v}{d}\rt)}{-s'} \notag \\
&=(2d)^{-2\sp} \int_0^\infty \!\!\! \f|_{[\xi]_k}|_{\lt[\bsm 0 & -1 \\ 4& 0 \esm\rt]}\lt(iy+\frac{v}{d}\rt)\vep_d^{-k}\kron{v}{d}(-2i)^\kt(diy)^\kt y^{-s'} \dy  \notag \\
&=(2d)^{1-2s} \vep_d^{-k} \kron{v}{d} \int_0^\infty \!\!\! \f_0\lt(iy+\frac{v}{d}\rt) y^{1-s + \kf} \dy 
\end{align}
where $\xi = \lt(\gamma, j(\gamma,z)\rt)$ and $\gamma = \bsm e & u \\ 4v & d \esm$. Thus for $2\nmid d$, 
\beq
\Lam{}{u}{d}{s} = (2d)^{1-2s}\,\vep_d^{-k} \kron{v}{d} \Lam{0}{v}{d}{1-s},
\eeq
which completes our proof.
\end{proof}

We now apply our functional equations to the double sum 
\beq
\addcond{d|r^2}{} \addcond{\substack{(u,d)=1\\ \runmod{u}{d}}}{ \Lambda(\f, \frac{u}{d}, s)}
\eeq
 from \eqref{main} in order to get an asymptotic bound for $D_r(s)$ at $\Re(s) < 0$ in terms of $r$. We first split the sum into appropriate parts.

\bal 
\addcond{d|r^2}{} &\addcond{\substack{(u,d)=1\\ \runmod{u}{d}}}{ \Lambda(\f, \frac{u}{d}, s)} \notag \\
& = \overbrace{\addcond{\substack{d|r^2\\ 4|d}}{} \addcond{\substack{(u,d)=1\\ \runmod{u}{d}}}{ \Lambda(\f, \frac{u}{d}, s)}}^{S_\infty} 
+ \overbrace{\addcond{\substack{d|r^2\\ 2||d}}{} \addcond{\substack{(u,d)=1\\ \runmod{u}{d}}}{ \Lambda(\f, \frac{u}{d}, s)}}^{S_\hf} 
+ \overbrace{\addcond{\substack{d|r^2\\ 2\nmid d}}{} \addcond{\substack{(u,d)=1\\ \runmod{u}{d}}}{ \Lambda(\f, \frac{u}{d}, s)}}^{S_0}. \label{butcher}
\end{align}
In this expression, $d$ can be assumed cube-free, since $r$ can be taken to be square-free and $d$ ranges over $d|r^2$.

Now we estimate this sum for $s$ slightly to the left of the line $\Re(s) = 0$. From \eqref{polyb} we have
\beq 
\Lam{*}{v}{d}{1- (-\vep)+i\tau} =O_{\vep,\f}((1+|\tau|)^{\frac{k}{4}+\vep}e^{-\frac{\pi}{2}|\tau|}),
\eeq
 where the implied constant is uniform over all $v$ and $d$, but is dependent on $\vep$ and $\f$. Using this along with Lemma \ref{class} and \eqref{butcher}, we get that, for $\vep>0$,

\begin{align}
\addcond{d|r^2}{} &\addcond{\substack{(u,d)=1\\ \runmod{u}{d}}}{ \Lambda\lt(\f, \frac{u}{d},-\vep+i\tau\rt)}\notag \\ 
& \ll_{\vep,\f} (1+|\tau|)^{\frac{k}{4}+\vep}e^{-\frac{\pi}{2}|\tau|}\sum_{d|r^2} \varphi(d) d^{1 - (-2\vep)} \notag \\
& \ll_{\vep,\f}(1+|\tau|)^{\frac{k}{4}+\vep}e^{-\frac{\pi}{2}|\tau|}\sigma_{2+2\vep}(r^2)\notag\\
& \ll_{\vep,\f}(1+|\tau|)^{\frac{k}{4}+\vep}e^{-\frac{\pi}{2}|\tau|}r^{4+5\vep}.
\end{align}
Thus using this estimate in \eqref{main},
\beq
D_r(-\vep+i\tau) \ll_{\vep,\f} (1+|\tau|)^{1+2\vep}r^{2+5\vep} .
\eeq 
Using this along with Lemma \ref{oneoverrsq}, a Phragm\'{e}n-Lindel\"{o}f argument tells us that 
\beq 
D_r(3/4+\vep+i\tau) \ll_{\vep,\f,\tau} 1/r^{1 + 4\vep}
\eeq 
which, when put to use in \eqref{first}, provides
\beq
M(s) = \sum_{\textrm{ square-free}}\frac{a(t)}{t^s} 
= \addtoinf{r}{\mu(r)} D_r(s) \ll_{\vep,\f,\tau}  \addtoinf{r}\frac{1}{r^{1+4\vep}} < \infty,
\eeq
where $s=\sigma+i\tau$ with $\sigma >3/4+\vep$. 
Therefore, we have proven the following:
\begin{prop} \label{analcont}
The series 
\beq
M(s) = \sum_{t} \frac{a(t)}{t^s}
\eeq
converges in the half plane $\Re(s)>\frac{3}{4}$ and also has only polynomial growth in $\im(s)$ in the vertical strips in that region.
\end{prop}
This was the desired pole-free region to prove Theorem \ref{mainthm}.

\section{Acknowledgments}
We would like to thank Jeff Hoffstein for introducing us to this problem and for suggesting how we might approach it.


\begin{thebibliography}{9}

\bibitem[1]{gh}\label{goldfeldhoffstein} Goldfeld, Dorian and Hoffstein, Jeffrey: \emph{Eisenstein Series of 1/2-integral Weight and the Mean Value of Real Dirichlet L-series.}  Invent. Math. 80 (1985, no. 2), 185-208

\bibitem[2]{hr}\label{hoffsteinrosen} Hoffstein, Jeffrey and Rosen, Michael: \emph{Average Values of L-series in Function Fields.}  J. Reine Angew. Math. 426 (1992), 117-150.

\bibitem[3]{iw} \label{iwaniec} Iwaniec, Henryk: \emph{Fourier coefficients of modular forms of half-integral weight.} Invent. Math. \textbf{87} (1987) 385-401.

\bibitem[4]{IK} \label{iwanieckowalski} Iwaniec, Henryk and Kowalski, Emmanuel: \emph{Analytic Number Theory}. Providence: American Mathematical Society. (2004)

\bibitem[5]{k}\label{krecent} Kohnen, Winfried: \emph{A Short Note on Fourier Coefficients of Half-Integral Weight Modular Forms}. International Journal of Number Theory. \textbf{6} (2010) 1255-1259.

\bibitem[6]{kz} \label{kohnen} Kohnen, W. and Zagier, D.: \emph{Values of $L$-series of modular forms at the center of the critical strip.} Invent. Math. \textbf{64} (1981) 175-198.

\bibitem[7]{nw} \label{niwa} Niwa, S.: \emph{On modular forms of half-integral weight and the integral of certain theta-functions.} Nagoya Math. J. \textbf{56} (1974) 147-161.

\bibitem[8]{sh} \label{shimura} Shimura, G.: \emph{On modular forms of half-integral weight.} Ann. Math. \textbf{97} (1973) 440-481. 

\bibitem[9]{w} \label{waldspurger} Waldspurger, J.-L., \emph{Sur les coefficients de Fourier des formes modulaires de poids demi-entier}, J. Math Pures Appl. \textbf{60} (1981) 375-484.
\end{thebibliography}
\end{document}